\documentclass[reqno,12pt]{amsart}%
\usepackage{amsfonts}
\usepackage{amsmath}
\usepackage{amssymb}

\def\bE{{\mathbb{E}}}

\def\bR{{\mathbb{R}}}
\def\bRd{{\mathbb{R}}^{\mathrm{d}}}

\def\cF{{\mathcal{F}}}
\def\cG{{\mathcal{G}}}

\def\cC{{\mathcal{C}}}
\def\cO{{\mathcal{O}}}

\def\cS{{\mathcal{S}}}

\def\dd{{\mathrm{d}}}

\newcommand\Laplace{\boldsymbol{\Delta}}

\newcommand\mfk[1]{\mathfrak{#1}}

\def\mfi{\mathfrak{i}}

\def\ds{\displaystyle}

\def\wc{\overset{d}{=}}

\newtheorem{theorem}{Theorem}
\theoremstyle{plain}
\newtheorem{corollary}[theorem]{Corollary}
\newtheorem{definition}[theorem]{Definition}
\newtheorem{proposition}[theorem]{Proposition}

\newtheorem{remark}[theorem]{Remark}

\newcommand\bel[1]{\begin{equation}\label{#1}}
\newcommand\ee{\end{equation}}

\numberwithin{theorem}{section}
\numberwithin{equation}{section}

\begin{document}
\title[Free Fields and SPDEs]
{Gaussian Fields and Stochastic Heat Equations}

\author{S. V. Lototsky}
\curraddr[S. V. Lototsky]{Department of Mathematics, USC\\
Los Angeles, CA 90089}
\email[S. V. Lototsky]{lototsky@usc.edu}
\urladdr{https://dornsife.usc.edu/sergey-lototsky/}

\author{A. Shah}
\curraddr[A. Shah]{Department of Mathematics, USC\\
	Los Angeles, CA 90089}
\email[A. Shah]{apoorvps@usc.edu}
\urladdr{}

\subjclass[2010]{Primary 60H15; Secondary 35R60, 60H40}

 \keywords{Cameron-Martin space; Cylindrical Brownian motion; Green's function; Homogenous distributions; Hermite heat equation}

\begin{abstract}
The objective of the paper  is to characterize the Gaussian free field as a stationary solution of the heat equation
with additive space-time white noise. In the case of $\bRd$, the investigation leads to  other types of
Gaussian fields, as well as interesting phenomena  in  dimensions one and two.
\end{abstract}

\maketitle

\today

\section{Introduction}
It is well-known, for example by the Donsker theorem \cite[Corollary VII.3.11]{JS-Lim},
 that a suitably  scaled simple symmetric random walk on $[0,1]$ converges to the standard Brownian motion.
When pinned (conditioned to hit zero) at the right point $x=1$, the same random walk converges to the Brownian bridge $\bar{W}=\bar{W}(x)$,
a Gaussian process on $[0,1]$ with mean zero and covariance
$$
\bE\big( \bar{W}(x)\bar{W}(y)\big)=\min(x,y)-xy;
$$
cf. \cite[Chapter VI]{Levy}.

What would a multi-dimensional version of these results be? In other words, what are the  scaling  limits of discrete random  objects in the plane or
in the space, or in higher dimensions?

In many models \cite[etc.]{Ken-GFF, PS-GFF} this limiting object is a Gaussian free field \cite{S-GFF}. While
well-known  in theoretical physics, for example, as a starting point in the construction of certain quantum field theories
\cite{Simon-qft}, the Gaussian free field  is a relatively new  area  of research in mathematics.

Let $\mathcal{O}\subseteq \bRd$ be a  domain and let $\Phi_{\mathcal{O}}=\Phi_{\mathcal{O}}(x,y),\ x,y\in \cO,$ be
Green's function of the Laplacian $\Laplace$ in $\mathcal{O}$ with suitable homogeneous  boundary conditions.
A {\em Gaussian free field} on $\mathcal{O}$ is usually defined as a (generalized) Gaussian process $\bar{W}=\bar{W}(x),\ x\in \cO$,
such that
\bel{GFF-Gen}
\bE\bar{W}(x)=0,\ \bE\Big(\bar{W}(x)\bar{W}(y)\Big)=\Phi_{\mathcal{O}}(x,y),\ \ x,y\in \cO.
\ee
If $\dd>1$, then  the function $ \Phi_{\mathcal{O}} $ has a singularity on the diagonal $x=y$. By \eqref{GFF-Gen},
 $\bE|\bar{W}(x)|^2=+\infty$ for all $x\in \cO$, meaning that  $\bar{W}$  must indeed be
a generalized process, or a random generalized function (distribution),  indexed by test functions on $\cO$ rather than points in $\cO$.

Let us assume that the equation
\bel{LE-0}
\Laplace v =-f
\ee
is well-posed in a sufficiently rich class $\cG$ of functions $f$ on $\cO$ and the
 solution of \eqref{LE-0} can be written as
$$
v(x)=\int_{\cO} \Phi_{\cO}(x,y) f(y)\, dy.
$$
 Then the Gaussian free field $\bar{W}$ on $\cO$ is defined as
  a collection of zero-mean Gaussian random variables $\bar{W}[f],\ f\in \cG$,
such that
\bel{GFF-Gen-ii}
\bE\Big( \bar{W}[f]\bar{W}[g]\Big)=\iint\limits_{\cO\times\cO} \Phi_{\cO}(x,y) f(x)g(y)\, dxdy,\ \ f,g\in \cG.
\ee
If  $\bar{W}=\bar{W}(x),\ x\in \cO,$ is a collection of Gaussian random variables satisfying \eqref{GFF-Gen},
then $\bar{W}$ defines a random distribution on $\cG$ by
$$
\bar{W}[f]=\int_{\cO} \bar{W}(x)f(x)\, dx,\ \ f\in \cG,
$$
which is  a collection of  zero-mean Gaussian random variables satisfying  \eqref{GFF-Gen-ii}.

Let $\mathbb{F}=\big(\Omega, \cF, \{\cF_t\}_{t\geq 0}, \mathbb{P}\big)$ be a stochastic basis with the usual assumptions \cite[Definition I.1.3]{JS-Lim}, on which
  countably many independent standard Brownian motions $w_k=w_k(t),
 \ t\geq 0,\ k=1,2,\ldots$ are defined. The stochastic basis $\mathbb{F}$ will be fixed throughout the rest of the paper.

  The  {\em space-time Gaussian white noise} $\dot{W}=\dot{W}(t,x)$ on $\cO$ is a collection of zero-mean
 Gaussian random variables  $\dot{W}[f]$,
 $f\in L_2\big((0,+\infty)\times \cO),$  such that
 $$
 \bE\Big(\dot{W}[f]\dot{W}[g]\Big)=\int_0^{+\infty} \int_{\cO} f(t,x)g(t,x)\, dx dt.
 $$
  Given an orthonormal basis $\{\mfk{h}_k=\mfk{h}_k(x),\ k\geq 1\}$ in $L_2(\cO)$,  the process $\dot{W}$ can be written as a (formal)
  sum
  \bel{STWN}
  \dot{W}(t,x)=\sum_{k=1}^{\infty} \mfk{h}_k(x)\dot{w}_k(t).
  \ee
  Similarly,
  \bel{CBM}
  W(t,x)=\sum_{k=1}^{\infty} \mfk{h}_k(x)w_k(t)
  \ee
  is called {\em cylindrical Brownian motion} on $L_2(\cO)$.
  For a square integrable function $f=f(t,x)$,
  $$
  \int_0^t \int_{\cO} f(s,y)W(ds,dy)=\sum_{k=1}^{\infty} \int_0^t\left(\int_{\cO} f(s,y)\mfk{h}_k(y)\, dy \right) dw_k(s);
  $$
  cf. \cite[Chapter 2]{Walsh}.

The objective of this paper is to characterize the Gaussian free field as the stationary solution of a heat equation driven by space-time
Gaussian white noise.
\begin{theorem}
\label{th:main-i}
Let $u=u(t,x)$ be a solution of
\bel{HE-main-i}
u_t(t,x)=\nu\,\Laplace u(t,x)+\sigma\dot{W}(t,x),\ t>0, \ x\in \cO\subseteq \bRd,
\ee
with  initial  condition $u(0,x)=\varphi(x)$ independent of $W$
and with constant $\nu>0$, $\sigma>0$; suitable boundary conditions are imposed if $\cO\subset \bRd.$

Then, as $t\to +\infty$, $u$ converges weakly  to a scalar multiple of the Gaussian free field on $\cO$.
\end{theorem}

In other words, as $t\to +\infty$, the solution of the stochastic parabolic equation \eqref{th:main-i} converges
in distribution to the solution of the stochastic elliptic equation
$$
(-\nu \Laplace)^{1/2} v(x) = \sigma {V}(x),
$$
where  ${V}$  is Gaussian white noise (or an isonormal Gaussian process) on $L_2(\cO)$.
By comparison, direct computations show that,  as $t\to +\infty$,
  the solution of the deterministic heat equation
$u_t=\nu\Laplace u + f(x)$ in a bounded domain or in $\bRd$, $d\geq 3$,
with a smooth compactly supported  $f$,
converges to the solution of the elliptic equation
$\nu\Laplace v=-f$, but this convergence does not in general hold in $\bR$ and $\bR^2$.

Here is an outline of the  proof of Theorem \ref{th:main-i}.
Denote by $G_{\cO}=G_{\cO}(t,x,y)$ the heat kernel for equation \eqref{HE-main-i} so that, for  $f\in \cG,$
the solution
$u^{{\mathrm{H}},f}=u^{{\mathrm{H}},f}(t,x)$  of the   deterministic heat equation
with initial condition $f$ is
\bel{D-Heat}
u^{{\mathrm{H}},f}(t,x)=\int_{\cO} G_{\cO}(t,x,y)f(y)dy.
\ee
If it exists,  the   solution of \eqref{HE-main-i} is
\bel{HE-main-i-sol}
u(t,x)=u^{{\mathrm{H}},\varphi}(t,x)+ \sigma \int_0^t \int_{\cO} G(t-s,x,y)\,W(ds,dy)
\ee
and, because $G_{\cO}(t,x,y)=G_{\cO}(t,y,x)$,
\begin{align}
\notag
u[t,f]&:=\!\int_{\cO}u(t,x)f(x) dx\!=\!u^{{\mathrm{H}},\varphi}[t,f]\!+\!
  \sigma \!\int_0^t\! \int_{\cO}\!\left(\int_{\cO}\!G_{\cO}(t-s,x,y)f(x)dx\right)W(ds,dy)\\
  \label{Mild-Gen}
& =u^{{\mathrm{H}},\varphi}[t,f]+ \sigma \int_0^t\int_{\cO} u^{{\mathrm{H}},f}(t-s,y)\,W(ds,dy);
 \end{align}
  cf. \cite[Chapter 9]{Walsh} in the case $\cO=\bRd,\ \dd\geq 3$.
As a result,
$$
\bE\Big(u[t,f]u[t,g]\Big)=\bE\Big(u^{{\mathrm{H}},\varphi}[t,f]u^{{\mathrm{H}},\varphi}[t,g]\Big)
+\sigma^2\!\int_0^{t}\!\int_{\cO}u^{{\mathrm{H}},f}(t-s,y)u^{{\mathrm{H}},g}(t-s,y)\, dy\, ds,
$$
and if  
\bel{Heat-zero}
\lim_{t\to +\infty} u^{{\mathrm{H}},f}(t,x)=0
\ee
in an appropriate way, then 
$$
\lim_{t\to +\infty} \bE\Big(u[t,f]u[t,g]\Big)=
\sigma^2\int_0^{+\infty}\int_{\cO}u^{{\mathrm{H}},f}(s,y)u^{{\mathrm{H}},g}(s,y)\, dy\, ds.
$$
Moreover, by \eqref{D-Heat} and the semigroup property of $G_{\cO}$,
$$
\int_{\cO}u^{{\mathrm{H}},f}(s,y)u^{{\mathrm{H}},g}(s,y)\, dy=
\iint\limits_{\cO\times\cO} G_{\cO}(s,x,y)f(x)g(y)\, dxdy.
$$
If we also have
\bel{Heat-Poisson}
\int_0^{+\infty}  G_{\cO}(s,x,y)\,ds=\frac{2}{\nu}\Phi_{\cO}(x,y),
\ee
then, combining the above computations with \eqref{GFF-Gen-ii}, we get the convergence
\bel{Conv000}
\lim_{t\to +\infty} \bE\Big(u[t,f]u[t,g]\Big)=\frac{2\sigma^2}{\nu} \bE\Big(\bar{W}[f]\bar{W}[g]\Big).
\ee

A major part of the paper consists in providing the details in the above arguments, in particular,
\begin{enumerate}
\item Constructing the  solution of \eqref{HE-main-i} and   interpreting \eqref{HE-main-i-sol};
\item Identifying a suitable function class $\cG$ and verifying \eqref{Mild-Gen}, \eqref{Heat-zero};
\item Working around \eqref{Heat-Poisson}: this step turns out to be a major technical difference between  a bounded domain
and the whole space;
\item Interpreting both $u$ and $\bar{W}$ as Gaussian measures on a suitable Hilbert space so that
\eqref{Conv000} will indeed imply the required convergence.
\end{enumerate}
We will also see that, similar to the deterministic problem,
 the cases $\cO=\bR$ and $\cO=\bR^2$ require special considerations,
partly because of the failure of \eqref{Heat-Poisson} and partly because of   unexpected difficulties interpreting \eqref{GFF-Gen-ii}.

Section \ref{Sec:GM} summarizes the construction and general properties of Gaussian processes indexed by elements of
a separable Hilbert space, which, in particular, provides an interpretation of the diverging series \eqref{STWN} and \eqref{CBM}.
Sections \ref{Sec:BD} and \ref{Sec:Rd}  present the precise statement and  proof of Theorem \ref{th:main-i}
 in a bounded domain and in the whole space,  respectively. Section \ref{Sec:OneD} discusses the special features of the one-dimensional
 case, and Section \ref{Sec:Sum} summarizes the results and puts them in a broader context.

The symbol $\sim$ has the same meaning as in \cite[Formula 2.1.1]{NIST}:
$$
f(x)\sim g(x),\ x\to x_0 \ \Leftrightarrow\ \lim_{x\to x_0} \frac{f(x)}{g(x)} =1.
$$

A long bar $\overline{z}$ over a symbol denoted complex conjugations; it should not be confused with a short
bar $\bar{W}$ in the notation of the Gaussian free field.

\section{Gaussian Processes and Measures on Hilbert Spaces}
\label{Sec:GM}
Let $H$ be a real separable Hilbert space with inner product $(\cdot, \cdot)_0$ and  norm $\|\cdot\|_0$, and
let $\Lambda$ be  a linear operator on $H$ with the following properties:
\begin{enumerate}
\item[[O1]\!\!] $(\Lambda f,g)_0=(f,\Lambda g)$ for all $f,g$ in the domain of $\Lambda$;
\item[[O2]\!\!] $(\Lambda f,f)_0>0,\ f\not=0$, $f$ in the domain of $\Lambda$;
\item[[O3]\!\!] There is an orthonormal basis $\{\mfk{h}_k,\ k\geq 1\}$ in $H$ such that
\bel{Eig000}
\Lambda \mfk{h}_k=\lambda_k \mfk{h}_k,\ k\geq 1;\ 0<\lambda_1\leq \lambda_2\leq \lambda_3\leq \cdots; \ \lim_{k\to \infty}
k^{-\alpha} \lambda_k =c_{\Lambda}
\ee
for some $\alpha>0$, $c_{\Lambda}>0$.
\end{enumerate}

For $f\in H$, write
$$
f_k=(f,\mfk{h}_k)_0.
$$
\begin{definition}
\label{def-HS}
 The {\tt Hilbert scale} $\mathbb{H}_{\Lambda}$ generated by the operator $\Lambda$ is the collection of
 the Hilbert spaces $\{ H^{\gamma},\ \gamma\in \bR\}$, where
\begin{itemize}
\item $H^0=H$;
\item $H^{\gamma}=\{f\in H: \sum_{k\geq 1} \lambda_k^{2\gamma} f_k^2<\infty\}$ if  $\gamma>0$;
\item $H^{\gamma}$ is the closure of $H$ with respect to the norm $\|f\|_{\gamma}$, where
\bel{Norm000}
\|f\|_{\gamma}^{2}=\sum_{k\geq 1} \lambda_k^{2\gamma} f_k^2,
\ee
if $\gamma<0$.
\end{itemize}
\end{definition}

Equality \eqref{Norm000} defines the norm in every $H^{\gamma}$, $\gamma\in \bR$,
$$
H^{\gamma}=\Lambda^{-\gamma}H,\ (f,g)_{\gamma}=(\Lambda^{\gamma} f, \Lambda^{\gamma} g)_0
=\sum_{k=1}^{\infty}\lambda_k^{2\gamma}f_kg_k,
$$
and
$$
f=\sum_{k=1}^{\infty} f_k\mfk{h}_k\in H^{\gamma} \ \Longleftrightarrow\ \sum_{k=1}^{\infty} k^{2\alpha\gamma}f_k^2<\infty.
$$
\begin{proposition}
\label{HS-gen}
If $\mathbb{H}_{\Lambda}=\{H^{\gamma},\ \gamma\in \bR\}$ is the Hilbert scale from Definition \ref{def-HS}, then,
for every $\gamma_1>\gamma_2$, the space $H^{\gamma_1}$ is densely and compactly embedded  into $H^{\gamma_2}$;
the embedding is Hilbert-Schmidt if $\gamma_1-\gamma_2>1/(2\alpha)$.
\end{proposition}
\begin{proof}
The construction of $\mathbb{H}_{\Lambda}$ implies  density of the embedding, whereas
assumption \eqref{Eig000} about the eigenvalues of $\Lambda$ implies that the embedding is compact and,
as long as $\sum_k \lambda_k^{2(\gamma_2-\gamma_1)}<\infty$,
it is  Hilbert-Schmidt.
\end{proof}

 \begin{definition}
 \label{def:CBM}
 Let $U$ be a separable Hilbert space with inner product $(\cdot, \cdot)_U$.
 \begin{enumerate}
\item  A {\tt $Q$-Brownian motion} $W=W(t)$ on $U$ is a collection of zero-mean Gaussian processes $\{W[t,h],\ h\in H,\ t\geq 0\}$
 such that $\bE\Big(W[t,h]W[s,g]\Big)=\min(t,s)\,(Qh,g)_{U}$ for some linear operator $Q$ on $U$. In the case $Q$ is the
 identity operator, $W$ is called a {\tt cylindrical Brownian motion} on $U$.
 \item A {\tt $Q$-Brownian motion} $W=W(t)$ on $U$ is called $U$-valued if
 \bel{CBM-val}
 W[t,h]=\big(W(t),h\big)_U
 \ee
 and the process $W$ on the right-hand side of \eqref{CBM-val} satisfies
  $$
  W\in L_2\Big(\Omega; \cC\big((0,T); U\big)\Big)
  $$
  for all $T>0$.
 \end{enumerate}

  \end{definition}

A $Q$-Brownian motion on $U$ is $U$-valued if and only if the operator $Q$ is trace class on $U$;
cf.  \cite[Propositions 4.3 and 4.4]{DaPr1}.
It is convenient to re-state \cite[Proposition 4.7]{DaPr1} in the setting of the Hilbert scale $\mathbb{H}_{\Lambda}$.

 \begin{proposition}
 \label{prop:CBM}
 A cylindrical Brownian motion   on $H$ has a representation
 \bel{CBM0}
 W(t)=\sum_{k\geq 1} \mfk{h}_k w_k(t),
 \ee
 where $w_k(t)=W[t,\mfk{h}_k],\ k\geq1,$ are independent standard Brownian motions, and
 $W\in L_2\Big(\Omega; \cC((0,T); H^{-\gamma})\Big)$ for all $T>0,\ \gamma>1/(2\alpha)$. Equivalently,
 a  cylindrical Brownian motion  $H$ is an $H^{-\gamma}$-valued
 $Q$-Gaussian process,  $\gamma>1/(2\alpha)$, with $Q=\mfk{j}\mfk{j}'$, where
 $\mfk{j}$ is the embedding operator $H\to H^{-\gamma}$ and $\mfk{j}': H^{-\gamma}\to H$ is the  adjoint of $\mfk{j}$.
 \end{proposition}

We will also need a stationary version of Definition \ref{def:CBM}.

 \begin{definition}
\label{def:IGP}
Let $U$ be a separable Hilbert space with inner product $(\cdot, \cdot)_U$.
 \begin{enumerate}
\item A {\tt $Q$-Gaussian process} $W$ on $U$ is a collection of zero-mean Gaussian random variables $\{W[h],\ h\in H\}$
 such that $\bE\Big(W[h]W[g]\Big)=(Qh,g)_U$ for some linear operator $Q$ on $U$. In the case $Q$ is the
 identity operator, $W$ is called an {\tt isonormal Gaussian process}; {\rm cf. \cite[Definition 1.1.1]{Nualart}}.
 \item A {\tt $Q$-Gaussian process} $W$ on $U$ is called $U$-valued if
 \bel{IGP-val}
 W[h]=\big(W,h\big)_U
 \ee
 and the random variable  $W$ on the right-hand side of \eqref{IGP-val} satisfies
  $W\in L_2\big(\Omega;U)\big)$.
 \end{enumerate}
 \end{definition}

 A $Q$-Gaussian process on $U$ is $U$-valued if and only if the operator $Q$ is
 trace class on $U$; cf. \cite[Theorem 3.2.39]{LR-SPDE}. In the  Hilbert scale $\mathbb{H}_{\Lambda}$, we  have
 a  version of Proposition \ref{prop:CBM}.

 \begin{proposition}
 \label{prop:IGP}
 Given an $r\in \bR$, an  isonormal Gaussian process on $H^r$ has a representation
 $$
 W=\sum_{k\geq 1}\lambda_k^{-r} \mfk{h}_k \zeta_k,
$$
 where $\zeta_k=W[\mfk{h}_k],\ k\geq1,$ are iid  Gaussian random variables, and
 $W\in L_2(\Omega; H^{r-\gamma})$ for all $\gamma>1/(2\alpha)$. Equivalently, an  isonormal Gaussian process on
 $H^r$ is an $H^{r-\gamma}$-valued
 $Q$-Gaussian process for every  $\gamma>1/(2\alpha)$, with $Q=\mfk{j}\mfk{j}'$, where
 $\mfk{j}$ is the embedding operator $H^r\to H^{r-\gamma}$ and $\mfk{j}': H^{r-\gamma}\to H^r$ is its adjoint.
 \end{proposition}

\begin{proof}
This follows by direct computation after observing that the collection
$$
\{\lambda_k^{-r} \mfk{h}_k,\ k\geq 1\}
$$
 is  an orthonormal basis in $H^r$.
\end{proof}

\begin{remark}
\label{rem:dual}
 {\rm While every Hilbert space is self-dual, there is an alternative notion of duality in a Hilbert scale $\mathbb{H}_{\Lambda}$: for every
$\gamma_0\in \bR $ and every $\gamma>0$, the spaces $H^{\gamma_0+\gamma}$ and $H^{\gamma_0-\gamma}$ are dual relative to the
inner product in $H^{\gamma_0}$; the duality $\langle \cdot,\cdot\rangle_{\gamma_0,\gamma}$ is given by
\bel{dual-rg}
f\in H^{\gamma_0+\gamma}, g\in H^{\gamma_0-\gamma} \ \mapsto \langle f,g \rangle_{\gamma_0,\gamma}=
\lim_{n\to \infty} (f,g_n)_{\gamma_0},
\ee
where $g_n \in H^{\gamma_0}$ and $\lim_{n\to \infty} \|g-g_n\|_{\gamma_0-\gamma}=0$.
 With respect to $ \langle \cdot,\cdot\rangle_{0,|r|}$ duality, an isonormal Gaussian process on $H^r$ from
Proposition \ref{prop:IGP}  becomes an isonormal Gaussian process on $H^{-r}$.
Indeed, if  $r>0$, then, for $f\in H^{-r},$ we {\em define}
$$
\langle W, f\rangle_{0,r} = \sum_{k=1}^{\infty} \frac{f_k}{\lambda_k^r} \, \zeta_k
$$
so that
$$
\bE \Big( \langle W, f\rangle_{0,r}  \langle  W, g \rangle_{0,r} \Big) = \sum_{k=1}^{\infty} \frac{f_kg_k}{\lambda_k^{2r}}=
(f,g)_{-r}.
$$
The case $r<0$ is similar. }
\end{remark}

\begin{remark}
{\rm Let $V$ be an isonormal Gaussian process on $H$.
By direct computation, an isonormal Gaussian process  $W$ on $H^r$ is the unique
solution of the stochastic elliptic equation
\bel{ell-gen}
\Lambda^{r/2} W = V;
\ee
cf. \cite[Theorem 4.2.2]{LR-SPDE}. }
\end{remark}

By the Bochner-Minlos theorem \cite[Theorem 2.27]{DaPr2},
 a $U$-valued $Q$-Gaussian process $W$ defines a centered Gaussian measure $\mu_W$ on $U$ by
$$
\mu_W(A)=\mathbb{P}(W\in A),
$$
where $A$ is a Borel sub-set of $U$, and, for every $f\in U$,
$$
\int_U e^{\mfi (f,g)_{U}}\, d\mu_W(g)=\mathbb{E}e^{\mfi (W,f)_{U}}=\exp\left(-\frac{1}{2}(Qf,f)_U\right).
$$
The {\em Cameron-Martin space} of the measure $\mu_W$ is the collection of all $h\in U$ such that the
measure  $\mu_W^h$ defined by $\mu_W^h(A)=\mu_W(A+h)$ is equivalent to $\mu_W$ \cite[Section 2.4]{Bogachev-GM}.

\begin{proposition}
\label{prop:GM-HS}
Let $\mathbb{H}_{\Lambda}$ be the Hilbert scale from  Definition \ref{def-HS}.
If $W$ is an isonormal Gaussian process on $H^r$, then $W$ generates a Gaussian measure $\mu_W$ on every
$H^{r-\gamma}$ with $\gamma>1/(2\alpha)$, and the Cameron-Martin space of this measure is $H^r$.
\end{proposition}

\begin{proof} This is a combination of two results, \cite[Lemma 2.1.4 and Theorem 3.5.1]{Bogachev-GM},
in the Hilbert space setting.
\end{proof}

\section{Bounded Domain in $\bRd$}
\label{Sec:BD}

Let $\cO$ be a bounded domain in $\bRd$ and let $\Laplace$ be the Laplacian on $\cO$ with some homogeneous boundary conditions
so that
\begin{enumerate}
\item[[A1\!\!]] The eigenfunction $\mfk{h}_k,\ k\geq 1,$ of $\Laplace$ form an orthonormal basis in $L_2(\cO)$;
\item[[A2\!\!]] The eigenvalues $-\lambda_k^2,\ k\geq 1,$ of $\Laplace$  satisfy $0<\lambda_1<\lambda_2\leq \lambda_3\leq \cdots$, and
there exists a number  $c_{\cO}>0$ such that
\bel{Weyl0}
\lambda_k\sim c_{\cO}k^{1/\dd}.
\ee

\end{enumerate}
There are various sufficient conditions ensuring [A1] and [A2]: see, for example,   \cite[Section 1.1.7]{SafVas}.

Taking $H=L_2(\cO)$ and $\Lambda=(-\Laplace)^{1/2}$, we see that conditions [O1]--[O3] hold, with  $\alpha=1/\dd$,
and  we construct the  Hilbert scale $\mathbb{H}_{\Lambda}$ as in Definition \ref{def-HS}.
In particular,
$$
f=\sum_{k=1}^{\infty} f_k\mfk{h}_k \in H^{\gamma} \ \ \Longleftrightarrow \ \ \sum_{k=1}^{\infty} k^{2\gamma/\dd}f_k^2<\infty.
$$

\subsection{Green's Functions and Gaussian Free Fields}
For $\nu>0$, consider the heat equation
\bel{HE-bd0}
u_t(t,x)= f(x)+\nu \int_0^t \Laplace u(s,x)\, ds ,\ t\geq 0,\ x\in \cO,
\ee
and the  Poisson equation
\bel{PE-bd0}
\nu \Laplace v(x)=-g(x), \ \ x\in \cO.
\ee
Writing
$$
f(x)=\sum_{k=1}^{\infty}f_k\mfk{h}_k(x), \ g=\sum_{k=1}^{\infty} g_k\mfk{h}_k(x),\ u(t,x)=\sum_{k=1}^{\infty}
 u_k(t)\mfk{h}_k(x),\ v(x)=\sum_{k=1}^{\infty} v_k\mfk{h}_k(x),
$$
 we can solve equations \eqref{HE-bd0} and \eqref{PE-bd0}.

 \begin{proposition}
 \label{prop-bnd-det}
\begin{enumerate}
\item For every $f\in H^{\gamma}$, the unique solution of \eqref{HE-bd0} is
$$
u(t,x)= \sum_{k=1}^{\infty} e^{-\lambda_k^2\nu t} f_k \mfk{h}_k(x)= \int_{\cO}G_{\cO}(t,x,y) f(y)\, dy,
$$
where
$$
G_{\cO}(t,x,y)=\sum_{k=1}^{\infty} e^{-\lambda_k^2 \nu t} \mfk{h}_k(x)\mfk{h}_k(y).
$$
The operator norm of the heat semigroup
\bel{HS-bnd}
S_t: f\mapsto \int_{\cO}G_{\cO}(t,x,y) f(y)\, dy
\ee
is decaying exponentially in time on every $H^{\gamma}\!:$
\bel{HT-S-bnd}
\|S_tf\|_{\gamma} \leq e^{-\lambda_1 t} \|f\|_{\gamma}.
\ee
\item For every $g\in H^{\gamma}$, the unique solution of \eqref{PE-bd0}  is
$$
v(x)=  \sum_{k=1}^{\infty}\frac{g_k}{\lambda_k^2\nu}\, \mfk{h}_k(x)=\int_{\bRd} \Phi_{\cO}(x,y) g(y)\, dy,
$$
where
$$
\Phi_{\cO}(x,y)=\sum_{k=1}^{\infty}
 \frac{\mfk{h}_k(x)\mfk{h}_k(y)}{\lambda_k^2\nu}=\int_0^{+\infty}G_{\cO}(t,x,y) \, dt.
$$
In particular, equality \eqref{Heat-Poisson} holds.
\end{enumerate}
\end{proposition}

\begin{definition}
\label{def:GFF-bnd}
The $(\Laplace,\cO)$-Gaussian free field $\bar{W}$ is an isonormal Gaussian process on $H^1$.
\end{definition}

The point is that, in a bounded domain $\cO$, there are many different Gaussian free fields, depending on the
boundary conditions of the operator $\Laplace$. For example, with {\em zero boundary conditions}, we take
$f,g\in H^1$ and integrate by parts to find
$$
\bE\big(\bar{W}[f]\bar{W}[g]\big)=
(f,g)_1=(\Lambda f,\Lambda g)_0=-(f,\Laplace g)_0= -\int_{\cO}
f(x)\Laplace g(x)\, dx = (\nabla f,\nabla g)_0,
$$
which, for $\dd=2$, is the same as  \cite[Definition 2.12]{S-GFF}. More generally, by
\eqref{ell-gen},
$$
(-\Laplace)^{1/2}\bar{W}=V,
$$
where  $V$ is an isonormal Gaussian process on $L_2(\cO)$.

\begin{proposition}
\label{prop:GFF-bnd}
Under the assumptions {\rm [A1], [A2]}, the  $(\Laplace,\cO)$-Gaussian free field $\bar{W}$ has a representation
\bel{GFF-bnd000}
\bar{W}(x)=\sum_{k=1}^{\infty} \frac{\zeta_k}{\lambda_k}\, \mfk{h}_k(x),
\ee
with iid standard Gaussian random variables $\zeta_k$, and defines a centered Gaussian measure on
$H^{1-(\dd/2)-\varepsilon}$ for every $\varepsilon>0$; the Cameron-Martin space of this measure is $H^1$.
\end{proposition}

\begin{proof} This follows from Proposition \ref{prop:IGP} with $r=1$ and $ \alpha=1/\dd$.
\end{proof}

\subsection{Main Result}

Given $\nu>0$, $\sigma>0$, and a cylindrical Brownian motion $W$ on $L_2(\cO)$,
consider the evolution equation
\bel{HE-main-bd}
u(t)=\varphi+\nu \int_0^t\Laplace u(s)\, ds +\sigma{W}(t),\ t>0,
\ee
with  initial condition $\varphi$ independent of $W$.

\begin{definition}
Given $\varphi\in L_2(\Omega; H^r)$, a solution of \eqref{HE-main-bd} is an adapted process with values in
$ L_2\big(\Omega\times [0,T];H^{r+1}\big)\bigcap L_2\big(\Omega; \cC((0,T);H^{r}\big)$,  such that
equality \eqref{HE-main-bd} holds in $H^{r-1}$ for all $t\geq 0$ with probability one.
\end{definition}

\begin{theorem}
\label{th:main-bd}
If $\varphi\in L_2(\Omega; H^{-\gamma})$ and $\gamma>\dd/2$, then,
under assumptions {\rm [A1], [A2]}, equation
\eqref{HE-main-bd} has a unique solution and, for every $T>0$,
\bel{HES-bnd-est0}
\bE\sup_{0<t<T} \|u(t)\|_{-\gamma}^2 +
\bE \int_0^T \|u(t)\|_{1-\gamma}^2\, dt \leq C(\gamma, T)(1+\bE\|\varphi\|^2_{-\gamma});
\ee
$C=C(\gamma, T)$ is a number depending only on $T$ and $\gamma$.
Moreover,
\begin{enumerate}
\item For every $t>0$, $u(t)\in L_2\big(\Omega; H^{1-\gamma}\big)$ and
\bel{HES-bnd-mild}
u(t)=S_t \varphi + \sum_{k=1}^{\infty} \bar{u}_k(t) \mfk{h}_k,
\ee
where $S_t$ is the heat semigroup \eqref{HS-bnd} and $\bar{u}_k(t),\ k\geq 1,$ are independent Gaussian
random variables with mean zero and variance
\bel{HES-bnd-var}
\bE \bar{u}_k^2(t)=\frac{\sigma^2}{2\nu\lambda_k^2}\Big(1-e^{-2\nu\lambda_k^2 t}\Big).
\ee
\item As $t\to +\infty$, the $H^{1-\gamma}$-valued random variables $u(t)$
converge weakly to $\sigma(2\nu)^{-1/2}\bar{W}$, where $\bar{W}$ is the $(\Laplace,\cO)$-Gaussian free field.
\end{enumerate}
\end{theorem}

\begin{proof}
The first part of the theorem follows directly from \cite[Theorem 3.1]{Roz} after
the identifications
$$
A=\nu \Laplace, \  \mathbb{X}=H^{1-\gamma}, \
\mathbb{H}=H^{-\gamma}, \ \mathbb{X}'=H^{-\gamma-1},\ M(t)=W(t),
$$
because, by Proposition \ref{prop:CBM},
$$
W\in  L_2\big(\Omega; \cC((0,T);H^{-\gamma}\big),\ \gamma>\frac{\dd}{2}.
$$
To establish \eqref{HES-bnd-mild}, we write
$$
u(t)=\sum_{k=1}^{\infty} u_k(t)\mfk{h}_k
$$
and combine \eqref{HE-main-bd} with \eqref{CBM0} to get
$$
u_k(t)=\varphi_k- \nu\lambda_k^2 \int_0^t u_k(s)\, ds + \sigma w_k(t);
$$
recall that $\Laplace \mfk{h}_k=-\lambda_k^2 \mfk{h}_k$. Then
$$
u_k(t) = \varphi_ke^{-\nu\lambda_k^2t}+\bar{u}_k(t),
$$
where
$$
\bar{u}_k(t)=\sigma \int_0^t e^{-\nu \lambda_k^2(t-s)}dw_k(s).
$$
Next,
$$
\bE\bar{u}_k^2(t)=\sigma^2 \int_0^t e^{-2\nu \lambda_k^2(t-s)}\, ds,
$$
 and \eqref{HES-bnd-mild} follows. In particular,
\bel{Reg-gm+1-b}
 \bE\|u(t)\|_{1-\gamma}^2\leq \frac{\bE\|\varphi\|_{-\gamma}^2}{\nu t}+\frac{\sigma^2}{2\nu}
 \sum_{k=1}^{\infty} \lambda_k^{-2\gamma},
 \ee
 so that
 \bel{Reg-gm+1}
 u(t)\in L_2(\Omega; H^{1-\gamma}),\ \  t>0.
 \ee
 Note that \eqref{Reg-gm+1-b} cannot be used to establish  \eqref{HES-bnd-est0},
 whereas \eqref{HES-bnd-est0} does not necessarily imply  \eqref{Reg-gm+1}.

 Finally, \eqref{HES-bnd-var} implies that, as $t\to +\infty$,  each $\bar{u}_k(t)$ converges in distribution to
 $\sigma(2\nu)^{-1/2}(\zeta_k/\lambda_k)$, and $\zeta_k,\ k\geq 1,$ are iid standard Gaussian random variables.
 By \eqref{HT-S-bnd} and independence of $\bar{u}_k(t)$ for different $k$,  the process $u(t)$ converges in distribution to the $H^{1-\gamma}$-valued
 Gaussian random variable
 $$
 \bar{W}=\frac{\sigma}{\sqrt{2\nu}}\sum_{k=1}^{\infty} \frac{\zeta_k}{\lambda_k}\, \mfk{h}_k,
 $$
 which, by Proposition \ref{prop:GFF-bnd}, concludes the proof of the theorem.
\end{proof}

\begin{corollary}
\label{cor:main-bd}
\begin{enumerate}
\item  Equation \eqref{HE-main-bd} is ergodic and the unique invariant measure is the distribution of $\sigma(2\nu)^{-1/2}\bar{W}$ on $H^{1-\gamma}$.
\item If $\varphi\wc \sigma(2\nu)^{-1/2}\bar{W}$, then $u(t)\wc \sigma(2\nu)^{-1/2}\bar{W}$ for all $t>0$.
\item If $\bE \varphi_k=0$ for all $k$, then, for each $t>0$, the measure generated by $u(t)$ on $H^{1-\gamma}$ is absolutely continuous
with respect to the measure generated by $\sigma(2\nu)^{-1/2}\bar{W}$.
\end{enumerate}
\end{corollary}
  \begin{proof}
  The first two statements are an immediate consequence of  \eqref{HES-bnd-mild}.
  The third statement follows from a theorem of Kakutani \cite[Example 2.7.6]{Bogachev-GM}:
  two zero-mean Gaussian product measures are equivalent if and only if the corresponding {\em standard deviations}
  $m_k$, $n_k$ satisfy
  $$
  \sum_{k=1}^{\infty} \left(\frac{m_k}{n_k}-1\right)^2<\infty:
  $$
  in our case,
  $$
  m_k=n_k\big(1-e^{-2\nu\lambda_k^2 t}\big)^{1/2}.
  $$
\end{proof}

\section{The whole space $\bRd$}
\label{Sec:Rd}
There are two  special features of the bounded domain that are absent in the whole space:
\begin{itemize}
\item The operator $\Lambda$ generating the scale $\mathbb{H}_{\Lambda}$ commutes with the operator $\Laplace$ in the equations
\eqref{HE-bd0} and \eqref{PE-bd0} we want to solve, and has the property that $\Lambda^{-\gamma}$ is Hilbert-Schmidt on $H$
 for sufficiently large $\gamma>0$;
\item The assumption $\lambda_1>0$ ensures  \eqref{HT-S-bnd}, that is,  the operator norm of the heat semigroup decays exponentially in time.
\end{itemize}

As a result, despite its simple form, equation \eqref{HE-main-i} in $\bRd$ is not covered by such standard references as
\cite{Kr_Lp1} (because of the structure of the noise) and \cite{DpZ-CBM} (because of the particular form of the evolution operator).
Accordingly, we study  \eqref{HE-main-i} in $\bRd$ by  combining very general results from \cite{DaPr1} and
\cite{Roz} with very specific  computations  using \eqref{HE-main-i-sol}.

\subsection{Function Spaces} There are three families of spaces that appear in the analysis of partial
differential equations on $\bRd$:
\begin{enumerate}
\item {\tt Homogeneous Sobolev spaces}  $\dot{H}^{\gamma}$, $\gamma\in \bR$,   the collection of generalized
functions $f\in \cS'(\bRd)$ such that the Fourier transform $\hat{f}=\hat{f}(\xi)$ of $f$ is locally integrable and
\bel{Hom-SSP}
\|f\|^2_{\dot{H}^{\gamma}}:=\int_{\bRd} |\xi|^{2\gamma} |\hat{f}(\xi)|^2\, d\xi<\infty;
\ee
when $\gamma<\dd/2$, $\dot{H}^{\gamma}$ is also known as the {\em Riesz potential} space \cite{Samko-Riesz};
\item {\tt Nonhomogenous  Sobolev, or Bessel potential,  spaces}  $H^{\gamma}$, $\gamma\in \bR$,   the collection of  generalized
functions $f\in \cS'(\bRd)$ such that the Fourier transform $\hat{f}=\hat{f}(\xi)$ of $f$ is locally
square integrable and
\bel{SSP}
\|f\|^2_{H^{\gamma}}:=\int_{\bRd} (\varepsilon+|\xi|^{2})^{\gamma} |\hat{f}(\xi)|^2\, d\xi<\infty;
\ee
\item The Hilbert scale ${\mathbb{H}}_{\tilde{\Lambda}}=\{\tilde{H}^{\gamma},\ \gamma\in \bR\},$
constructed  according to Definition \ref{def-HS} with  $H=L_2(\bRd)$ and $\tilde{\Lambda}$ defined by
\bel{HH-op}
\tilde{\Lambda}^2: f(x)\mapsto -\Laplace f(x) + |x|^2f(x),\ f\in \cS(\bRd).
\ee
 The operator $\tilde{\Lambda}^2$ has pure point spectrum
so that \eqref{Eig000} holds with $\alpha=1/(2\dd)$,
and the   eigenfunctions, known as the Hermite functions, form an orthonormal basis in $L_2(\bRd)$; cf.
\cite[Section 1.5]{GJ} or \cite[Example 4.2]{Walsh}.
\end{enumerate}

Recall that the {\em normalized Hermite polynomials} are
$$
\mathrm{H}_n(x) = \frac{(-1)^n}{\pi^{1/4}2^{n/2}(n!)^{1/2}}\,e^{x^2}\frac{d^n}{dx^n} e^{-x^2}, \ \ n=0,1,2,\ldots;
$$
 the  {\em Hermite functions}
$$
h_n(x)=e^{-x^2/2}\mathrm{H}_n(x)
$$
form an orthonormal basis in $L_2(\bR)$ and satisfy
$$
-h_n''(x)+x^2h_n(x)=(2n+1)h_n(x).
$$
The  orthonormal basis in  $L_2(\bRd)$,
$$
h_{\mathbf{n}}(x_1,\ldots, x_{\dd})=\prod_{j=1}^{\dd} h_{n_j}(x_j),
$$
 is indexed by  $\mathbf{n}=(n_1, \ldots, n_{\dd})$, $n_j=0,1,2,\ldots$ so that
 $$
 \tilde{\Lambda}^2h_{\mathbf{n}}=\lambda_{\mathbf{n}}^2h_{\mathbf{n}}=\big(2(n_1+\cdots+n_{\dd}) + \dd\big)h_{\mathbf{n}}.
 $$
 A non-decreasing  ordering  of  $\lambda_{\mathbf{n}}^2$ brings us to  the setting of Definition \ref{def-HS}.  In particular,
 $$
 \lambda_n^2\sim (2\dd!)^{1/\dd}\, n^{1/\dd},
 $$
 cf. \cite[Theorem 30.1]{Shubin-PDO}, and
$$
f=\sum_{k=1}^{\infty} f_k\mfk{h}_k \in \tilde{H}^{\gamma} \ \ \Longleftrightarrow \ \ \sum_{k=1}^{\infty} k^{\gamma/\dd}f_k^2<\infty.
$$

The norms \eqref{SSP} are equivalent for different $\varepsilon>0$ and \eqref{Hom-SSP} is a formal limit of
\eqref{SSP} as $\varepsilon \to 0$.
We could interpret \eqref{Hom-SSP} and \eqref{SSP} as
$$
\dot{H}^{\gamma}=\dot{\Lambda}^{-\gamma} L_2(\bRd),\ \
H^{\gamma}=\Lambda^{-\gamma}L_2(\bRd),
$$
with
$$
\dot{\Lambda} =(-\Laplace)^{1/2}: \hat{f}(\xi)\ \mapsto |\xi|  \hat{f}(\xi),\ \
 {\Lambda} = (\varepsilon-\Laplace)^{1/2}: \hat{f}(\xi)\ \mapsto (\varepsilon+|\xi|^2)^{1/2} \hat{f}(\xi),
 $$
 but it is still not possible to construct the   scales as in  Definition \ref{def-HS}:
 the operators $\dot{\Lambda}$ and $\Lambda$ do not have a pure point spectrum, and, in addition,
 the spaces $\dot{H}^{\gamma}$ are complete with respect to the norm $\|\bullet\|_{\dot{H}^{\gamma}}$
 if and only if $\gamma<\dd/2$ \cite[Proposition 1.3.4]{Chemin-PDEs}. In particular, $\dot{H}^1$ is not a Hilbert
 space when $\dd=1,2$.

 It follows from the definitions that $H^{\gamma} \subset \dot{H}^{\gamma}$ and $\tilde{H}^{\gamma}\subset H^{\gamma}$
for $\gamma>0$, and $\dot{H}^{\gamma} \subset {H}^{\gamma}$ for $\gamma<0$. Also, by duality,
$H^{\gamma}\subset\tilde{H}^{\gamma}$ for $\gamma<0$. To summarize,
\bel{SSP-incl}
\begin{cases}
\tilde{H}^{\gamma}\subset H^{\gamma} \subset \dot{H}^{\gamma},&  \gamma>0,\\
\tilde{H}^{0}= H^{0}= \dot{H}^{0}=L_2(\bRd),& \gamma=0,\\
\dot{H}^{\gamma} \subset {H}^{\gamma}\subset\tilde{H}^{\gamma}, & \gamma<0.
\end{cases}
\ee

One of the technical difficulties in studying equation \eqref{HE-main-i} on $\bRd$ is that, while
the spaces  $\dot{H}^{\gamma}$ and $H^{\gamma}$ are ``custom-made'' for the operator $\Laplace$,
the cylindrical
Brownian motion $W=W(t)$ on $L_2(\bRd)$ does not belong to any of those space, even though
we do have $\psi W(t)\in H^{-\gamma}$, $\gamma>\dd/2$ for every $t>0$ and every
smooth function $\psi$ with compact support \cite[Proposition 9.5]{Walsh}. On the other hand,
by Proposition \ref{prop:CBM}, we have
\bel{CBM-Rd}
W\in L_2\big(\Omega;\cC((0,T);\tilde{H}^{-\gamma}\big),\ T>0,
\ee
for every $\gamma>\dd$, meaning that the basic existence/uniqueness result for \eqref{HE-main-i} must be
established in   $\tilde{H}^{\gamma}$. Another useful feature of the spaces $\tilde{H}^{\gamma}$ is the equalities
$$
\cS(\bRd)=\bigcap_{\gamma} \tilde{H}^{\gamma},\ \ \cS'(\bRd)=\bigcup_{\gamma} \tilde{H}^{\gamma};
$$
cf. \cite{BS-Sch}.

\begin{definition}
\label{Def-GFF-Rd}
The {\tt Gaussian free field}  $\bar{W}$ on $\bRd$, $\dd\geq 3$, is   an isonormal Gaussian process on $\dot{H}^1$.
  The {\tt Euclidean free field of mass $\sqrt{\varepsilon}$}  is an isonormal Gaussian process  $\bar{W}_{\varepsilon}$ on  $H^1$.
\end{definition}
We also denote by   $\tilde{W}$  an isonormal Gaussian process on $\tilde{H}^1$.

To state a  definition of   $\bar{W}$   that works for all $\dd$,
denote by $\cS_0(\bRd)$ the collection of functions from $\cS(\bRd)$ for which the Fourier transform is equal to zero near the origin.

\begin{definition}
\label{def:GFF-12}
The  {\tt Gaussian free field}  $\bar{W}$ on $\bRd$, $\dd\geq 1$, is a collection of zero-mean Gaussian random variables
$\bar{W}[f]$, $f\in \cS_0(\bRd)$ such that
\bel{eq:GFF-12}
\mathbb{E}\Big(\bar{W}[f]\bar{W}[g]\Big)=
\int_{\bRd} \frac{\hat{f}(\xi)\overline{\hat{g}(\xi)}}{|\xi|^2}\, d\xi.
\ee
\end{definition}
In the language of quantum field theory \cite[p. 103]{GJ}, construction of a zero-mass free field $(\varepsilon=0)$
  in dimensions one and two requires  different sets of test functions.

For $\dd\geq 3$,  Definitions \ref{Def-GFF-Rd} and \ref{def:GFF-12}
    are equivalent. Indeed,
 the space  $\cS_0(\bRd)$ is dense in $\dot{H}^{\gamma}$ for $\gamma<\dd/2$ \cite[Proposition 1.35]{Chemin-PDEs}
 and, for $|\gamma|<\dd/2$, the spaces $\dot{H}^{\gamma}$ and $\dot{H}^{-\gamma}$ are dual relative to the
 inner product of $L_2(\bRd)$ \cite[Proposition 1.36]{Chemin-PDEs}. Thus, if $\dd\geq 3$, then  the isonormal Gaussian process
   on $\dot{H}^{1}$ satisfies \eqref{eq:GFF-12} with an
   interpretation of $\bar{W}[f]$ as duality  relative to  $L_2(\bRd)$
   (as opposed to inner product in $\dot{H}^{1}$; cf. Remark \ref{rem:dual}).

Definitions  \ref{def:GFF-12} is also consistent with  \eqref{GFF-Gen}.
Indeed,  the function $\xi\mapsto |\xi|^{-2}$  is a homogeneous distribution in $\cS'(\bRd)$ and, for $\dd\not=2$, the
Fourier transform of this distribution is the fundamental solution of the Poission equation on $\bRd$; cf. \cite[Chapter 32]{Donoghue-FTransf}.
When $\dd=2$, there are some issues with uniqueness, which can be resolved, for example, by restricting the set of test
functions to $\cS_0(\bRd)$.

Finally, by \eqref{ell-gen}, if $V$ is an isonormal Gaussian process on $L_2(\bRd)$, then
$$
(-\Laplace)^{1/2}\bar{W}=V,\ \ (\varepsilon-\Laplace)^{1/2}\bar{W}_{\varepsilon}=V,\ \
\tilde{\Lambda}\tilde{W}=V.
$$

\subsection{Deterministic Equations and Fundamental Solutions}
For $\nu>0, \varepsilon\geq 0$ and $f\in \mathcal{S}(\bRd)$, consider the heat equation
\bel{HE0}
u_t(t,x)= \nu \Laplace u(t,x)-\varepsilon u(t,x),\ t>0,\ x\in \bRd,
\ee
with initial condition $u(0,x)=f(x)$,
and the   Poisson equation
\bel{PE0}
\nu \Laplace v(x)-\varepsilon v(x)=-g(x), \ \ x\in \bRd.
\ee
The number $\varepsilon>0$ in $\bRd$  is the analog of $\lambda_1>0$ in the bounded domain.

Below is a summary of the well-known results.
\begin{itemize}
\item The unique solution of \eqref{HE0} in $\mathcal{S}(\bRd)$ is
$$
u(t,x)=  \int_{\bRd}G_{\varepsilon,\dd}(t,x) f(y)\, dy,
$$
where
\bel{GF-H}
G_{\varepsilon,\dd}(t,x)=\frac{1}{(4\pi\nu t)^{\dd/2}}\exp\left(-\varepsilon t -\frac{|x|^2}{4 \nu t}\right);
\ee
cf. \cite[Theorem 8.4.2]{Kr-H}.
\item The unique solution of \eqref{PE0} in $\mathcal{S}(\bRd)$ is
\bel{PE-S-ep}
v(x)= \int_{\bRd} \Phi_{\varepsilon,\dd}(x-y) g(y)\, dy,
\ee
where
\bel{GF-L}
\Phi_{\varepsilon,\dd}(x)=\int_0^{+\infty}G_{\varepsilon,\dd}(t,x) \, dt;
\ee
cf. \cite[Theorems 1.2.1 and 1.6.2, and Exercise 1.6.5]{Kr-H}.
\item If $\varepsilon=0$ and $d\geq 2$, then the unique solution of \eqref{PE0} in $\mathcal{S}(\bRd)$ is
$$
v(x)= \int_{\bRd} \Phi_{0,\dd}(x-y) g(y)\, dy,
$$
where
$$
\Phi_{0,\dd}(x)=
\begin{cases}
\ds -\frac{1}{2\pi\nu} \ln \frac{|x|}{\sqrt{\nu}}\,,&\ \dd=2,\\
&\\
\ds \frac{\Gamma\left(\frac{\dd}{2}-1\right)}{4\pi^{\dd/2}\nu|x|^{\dd-2}}\,,& \dd\geq 3,
\end{cases}
$$
and
$$
\Gamma(x)=\int_0^{+\infty} t^{x-1} e^{-t}\, dt;
$$
cf. \cite[Section 2.2.1]{Evans}.
\end{itemize}

Denote by $K_{p}=K_{p}(x)$, $p, x\in \bR$, the modified Bessel function of the second kind \cite[Section 10.25]{NIST}.

\begin{proposition}[cf. \protect{\cite[Proposition 7.2.1]{GJ}}]
The following equalities hold:
\begin{align}
\label{Bessel1}
\Phi_{\varepsilon,\dd}(x)&=
(2\pi\nu)^{-\dd/2}\left(\frac{\varepsilon\nu}{x^2}\right)^{(\dd-2)/4}K_{(\dd-2)/2} \big(\sqrt{\varepsilon/\nu}\, |x|\big),\\
\label{Bessel2}
\lim_{\varepsilon\to 0}\Phi_{\varepsilon,\dd}(x)&= \Phi_{0,\dd}(x),\ x\in \bRd\setminus \{0\},\ \ \dd\geq 3.
\end{align}
In particular,
\bel{Phi123}
\Phi_{\varepsilon,\dd}(x)=
\begin{cases}
\ds \frac{1}{2\sqrt{\varepsilon\nu}}\,e^{-\sqrt{\varepsilon/\nu}\,|x|},&\ \dd=1,\\
&\\
\ds \frac{1}{2\pi\nu}\,K_0 \big(\sqrt{\varepsilon/\nu}\, |x|\big),&\ \dd=2,\\
&\\
\ds \frac{1}{4\pi\nu |x|}\,e^{-\sqrt{\varepsilon/\nu}\, |x|},&\  \dd=3.
\end{cases}
\ee
\end{proposition}

\begin{proof}
Equality \eqref{Bessel1}:  combine \eqref{GF-L} with \cite[Formula 3.471.9]{Gradshtein-Ryzhyk}.
Equality \eqref{Bessel2}: use the properties of the function $K_p$,
in particular, $K_{p}(x)=K_{-p}(x),\ \nu>0$ \cite[Formula 10.27.3]{NIST} and
$K_{p}(x)\sim 2^{p-1}\Gamma(p)x^{-p},\ x\to 0, \ p>0$  \cite[Formula 10.30.2]{NIST}. Of course, one can get
\eqref{Bessel2} directly by passing to the limit in \eqref{GF-L}. Equality \eqref{Phi123} is a particular
case of \eqref{Bessel1} because
$$
K_{\pm 1/2}(z)=\sqrt{\frac{\pi}{2 z}}\,e^{-z};
$$
cf.  \cite[Formula 10.39.2]{NIST}.
\end{proof}

Combining \eqref{Phi123} with  $K_0(x)\sim -\ln x,\ x\to 0$ \cite[Formula 10.30.3]{NIST}, we see
that, in the case $\dd=2$, equality  \eqref{Bessel2} is missed by a logarithmic term: for fixed $x\not=0$,
$$
\Phi_{\varepsilon,2}(x)\sim  \Phi_{0,2}(x)-\frac{1}{4\pi\nu} \ln \varepsilon,\ \varepsilon\to 0.
$$
 Representation \eqref{PE-S-ep} has a version in the Fourier domain, with no explicit dependence on $\dd$:
 $$
 \hat{v}(\xi)=\frac{\hat{f}(\xi)}{\varepsilon+\nu|\xi|^2}.
 $$
Passing to the limit $\varepsilon\to 0$, we get
$$
 \hat{v}(\xi)=\frac{\hat{f}(\xi)}{\nu|\xi|^2}.
$$
Event though  the function $\xi\mapsto |\xi|^{-2}$ is not integrable at zero for $\dd=1,2$,
it defines  a homogenous distribution on $\cS_0(\bRd)$, and its inverse
Fourier transform is equal to $\Phi_{0,\dd}$ \cite[Chapter 32]{Donoghue-FTransf}.

  To conclude, we summarize how the main operators act in the spaces $\tilde{H}^{\gamma}$.

\begin{proposition}
\label{prop:MO}
For every $\gamma\in \bR$,
\begin{itemize}
\item[(C1)] the operator $\Laplace$ extends to a bounded linear operator from
$\tilde{H}^{\gamma+2}$ to $\tilde{H}^{\gamma}$;
\item[(C2)] the heat semigroup
\bel{heat:Rd}
S_t: f(x)\mapsto \int_{\bRd} G_{\nu\varepsilon, \dd}(t,x-y)f(y)\, dy, \   \varepsilon\geq 0,\ t>0,\ f\in \cS(\bRd),
\ee
 extends to  a bounded linear operator on $\tilde{H}^{\gamma}$, and, if $\varepsilon>0$, then
 \bel{Heat-zero-ep}
 \|S_tf\|_{\tilde{H}^{\gamma}}\leq Ce^{-\delta t}\|f\|_{\tilde{H}^{\gamma}}
 \ee
 for some $C>0, \, \delta>0$ and all $f\in \tilde{H}^{\gamma}$.
\end{itemize}
\end{proposition}

\begin{proof}

(C1) Direct computations show that $\Laplace$ is bounded from $\tilde{H}^{2k+2}$ to $\tilde{H}^{2k}$
 for every $k=0,1,2,\ldots.$ The case of $\gamma>0$ then follows by interpolation and
 $\gamma<0$, by duality.

(C2) This follows by \cite[Theorem 2.4]{RT-Heat}.
\end{proof}

\subsection{Main Results}
Given $\nu>0$, $\sigma>0$, $\varepsilon >0$,  a cylindrical Brownian motion
$W$ on $L_2(\bRd)$, and $\varphi\in L_2(\Omega; \tilde{H}^{r})$ independent of $W$,
consider  stochastic evolution equations
\begin{align}
\label{HE-main-em}
u(t)&=\varphi+{\nu} \int_0^t (\Laplace-\varepsilon) u(s)\, ds +\sigma{W}(t),  \\
\label{HE-main-zm}
u(t)&=\varphi+{\nu} \int_0^t \Laplace u(s)\, ds +\sigma{W}(t), \\
\label{HE-main-hh}
u(t)&=\varphi+{\nu} \int_0^t \tilde{\Lambda}^2 u(s)\, ds +\sigma{W}(t),
\end{align}
with $\tilde{\Lambda}^2$ from \eqref{HH-op}. In physics literature, the
deterministic version of \eqref{HE-main-hh} is known as the Hermite heat equation  \cite{HermiteHeat1}.

\begin{definition}
\label{def:solRd}
For each of the three equations, given the initial condition $\varphi\in L_2(\Omega; \tilde{H}^r)$,
a solution $u=u(t)$ on $[0,T]$ is an adapted  process with values in
$L_2\big(\Omega\times(0,T);\tilde{H}^{r+1}\big)\bigcap L_2\big(\Omega;\cC((0,T);\tilde{H}^{\gamma}\big)$, such that
the corresponding equality holds in $\tilde{H}^{r-1}$ for all $t\in [0,T]$ with probability one.
\end{definition}

\begin{theorem}
\label{th:main-em}
Assume that $\varphi\in \tilde{H}^{-\gamma}$ and $\gamma>\dd$. Then
\begin{enumerate}
\item Equation \eqref{HE-main-em} has a unique solution  for every $T>0$;
\item The solution has a representation
\bel{Mild-em}
u(t)=S_t\varphi+\int_0^t S_{t-s}dW(s),
\ee
with $S_t$ from \eqref{heat:Rd};
 \item as $t\to +\infty$, the solution  converges in distribution to  $(2\nu)^{-1/2}\sigma W_{\varepsilon}$, that is, the
 Gaussian measure generated on $\tilde{H}^{-\gamma}$ by the solution converges weakly to the
 Gaussian measure generated by $(2\nu)^{-1/2}\sigma W_{\varepsilon}$.
 \end{enumerate}
\end{theorem}

\begin{proof}
The general theory of SPDEs in the Sobolev spaces $H^{\gamma}$, such as \cite{Kr_Lp1}, is not
applicable because the process $W=W(t)$ does not take values in any of $H^{\gamma}$. Similarly, the   results
from \cite{DpZ-CBM} do not apply because the operator $(-\Laplace)^{-1}$ is not Hilbert-Schmidt on $L_2(\bRd)$.

Fortunately,
for existence and uniqueness of solution,   relation \eqref{CBM-Rd} and  first part of Proposition \ref{prop:MO}
 make it possible to apply  \cite[Theorem 3.1]{Roz} with
$$
A=\nu(\Laplace-\varepsilon), \  \mathbb{X}=\tilde{H}^{1-\gamma}, \
\mathbb{H}=\tilde{H}^{-\gamma}, \ \mathbb{X}'=\tilde{H}^{-\gamma-1},\ M(t)=W(t).
$$
Similarly, the second part of  Proposition \ref{prop:MO} makes  it possible to apply \cite[Theorem 5.4]{DaPr1},
from which  \eqref{Mild-em} follows.

To prove convergence, note that, by  \eqref{GF-L}, the general argument outlined in Introduction works, with $\cO=\bRd$ and $\cG=\cS(\bRd)$.
Keeping in mind that the fundamental solution for \eqref{HE-main-em} is $G_{\nu\varepsilon, \dd}$, which, by \eqref{GF-H}, acts
as the multiplier
$$
\hat{f}(\xi) \ \mapsto e^{-\nu(|\xi|^2+\varepsilon)t}\hat{f}(\xi)
$$
 in the Fourier domain, we easily complete the proof.
\end{proof}

\begin{theorem}
\label{th:main-zm}
If $\varphi\in \tilde{H}^{-\gamma}$ and $\gamma>\dd$, then
equation \eqref{HE-main-zm} has a unique solution  for every $T>0$ and the
 solution has a representation
\bel{Mild-zm}
u(t)=S_t\varphi+\int_0^t S_{t-s}dW(s),
\ee
where $S_t$ is from \eqref{heat:Rd} with $\varepsilon = 0$.

If $\varphi\in {H}^{-\gamma}$ and $\gamma>\dd$, then,  as $t\to +\infty$, the solution  converges in distribution to
 $(2\nu)^{-1/2}\sigma \bar{W}$, that is, the
 Gaussian measure generated on $\tilde{H}^{-\gamma}$ by the solution converges weakly to the
 Gaussian measure generated by $(2\nu)^{-1/2}\sigma \bar{W}$.
\end{theorem}

\begin{proof}
Existence, uniqueness, and representation \eqref{Mild-zm} of the solution follow in the same way as  in the proof of Theorem \ref{th:main-em}.
To prove the convergence as $t\to +\infty$, we streamline the notations by setting  $G=G(t,x)$ to be the heat kernel for equation \eqref{HE-main-zm}:
$$
G(t,x)=\frac{1}{(4\nu t)^{\dd/2}}\, e^{-|x|^2/(4\nu t)}.
$$
 Given a function $f=f(x)$ from $\mathcal{S}_0(\bRd)$, denote by
$u^{{\mathrm{H}},f}=u^{{\mathrm{H}},f}(t,x)$ the  solution of the deterministic heat equation
with initial condition $f$:
\begin{equation*}
u^{{\mathrm{H}},f}(t,x)=\int_{\bRd} G(t,x-y)f(y)dy.
 \end{equation*}
Then
$$
u(t,x)=u^{{\mathrm{H}},\varphi}(t,x)+\sigma \int_0^t \int_{\bRd} G(t-s,x-y)\,W(ds,dy)
$$
and, using \eqref{dual-rg} and \cite[Theorem 5.4]{DaPr1},
\begin{equation*}
\begin{split}
u[t,f]&:=\langle f, u(t) \rangle_{0,\gamma} \\
&=
u^{{\mathrm{H}},\varphi}[t,f]+  \sigma \int_0^t \int_{\bRd}\left(\int_{\bRd}G(t-s,x-y)f(x)dx\right)\,W(ds,dy)\\
& = u^{{\mathrm{H}},\varphi}[t,f]+  \int_0^t\int_{\bRd} u^{{\mathrm{H}},f}(t-s,y)\,W(ds,dy).
  \end{split}
  \end{equation*}
By independence of $\varphi$ and $W$,
\begin{equation*}
\begin{split}
\bE\Big(u[t,f]u[t,g]\Big)\!=\!\bE\Big(u^{{\mathrm{H}},\varphi}[t,f]u^{{\mathrm{H}},\varphi}[t,g]\Big)\!+\!
\sigma^2\!\int_0^t\int_{\bRd}u^{{\mathrm{H}},f}(t-s,y)u^{{\mathrm{H}},g}(t-s,y)\, dy\, ds\\
=\bE\Big(u^{{\mathrm{H}},\varphi}[t,f]u^{{\mathrm{H}},\varphi}[t,g]\Big)+
\sigma^2\int_0^t\int_{\bRd}u^{{\mathrm{H}},f}(s,y)u^{{\mathrm{H}},g}(s,y)\, dy\, ds.
\end{split}
\end{equation*}
Next,
$$
\hat{u}^{{\mathrm{H}},f}(s,\xi)=\hat{f}(\xi)e^{-s\nu|\xi|^2},
$$
and then the Fourier isometry implies
\begin{equation}
\label{eq-aux-zm}
\begin{split}
\bE\Big(u[t,f]u[t,g]\Big)&=\iint\limits_{\bRd\times\bRd}e^{-t\nu(|\xi|^2+|\eta|^2)}
\bE\big(\hat{\varphi}(\xi)\hat{\varphi}(\eta)\big)\overline{\hat{f}(\xi)\hat{g}(\eta)}\,d\xi\,d\eta\\
&+
\sigma^2\int_0^t\int_{\bRd}
\hat{f}(\xi)\overline{\hat{g}(\xi)}e^{-2s\nu|\xi|^2}\, d\xi\, ds.
\end{split}
\end{equation}
The first term on the right-hand side of \eqref{eq-aux-zm} goes to zero as $t\to \infty$ by the dominated convergence theorem,
because, by assumption,
$$
\int_{\bRd}(1+|\xi|^2)^{-\gamma}\bE|\hat{\varphi}(\xi)|^2d\xi<\infty
$$
for some $\gamma>\dd$, and, for $f,g\in \cS_0(\bRd)$,
$$
\sup_{\xi}(1+|\xi|^2)^{\gamma}|\hat{f}(\xi)|<\infty,\ \
\sup_{\eta}(1+|\eta|^2)^{\gamma} |\hat{g}(\eta)|<\infty.
$$
With $\varepsilon=0$, we no longer have \eqref{Heat-zero-ep} and therefore have to make additional assumptions about the
initial condition to achieve the desired convergence.

As a result,
$$
\lim_{t\to +\infty}\bE\Big(u[t,f]u[t,g]\Big)=
\sigma^2\int_{\bRd}\frac{\hat{f}(\xi)\overline{\hat{g}(\xi)}}{2\nu |\xi|^2}\, d\xi.
$$
Together with  \eqref{eq:GFF-12},  the last equality completes the proof.
\end{proof}

Analysis of equation \eqref{HE-main-hh}  in the scale  $\mathbb{H}_{\tilde{\Lambda}}$
is equivalent to analysis of equation \eqref{HE-main-bd}
 in the scale $\mathbb{H}_{{\Lambda}}$:
 similar to Theorem \ref{th:main-bd} and Corollary \ref{cor:main-bd}, the
 distribution of $\sigma(2\nu)^{-1/2}\tilde{W}$ is  the unique invariant measure for equation
 \eqref{HE-main-hh}. The only difference is that now we have $\lambda_k$ of order $k^{1/(2\dd)}$
rather than $k^{1/\dd}$.

 Let
 $\mfk{h}_k,\ k\geq 1,$ be the Hermite functions and let $\lambda_k,\ k\geq 1,$ be the
 corresponding eigenvalues of the operator $\tilde{\Lambda}^2$.

\begin{theorem}
\label{th:main-hh}
Assume that $\varphi\in \tilde{H}^{-\gamma}$ and $\gamma>\dd$. Then
\begin{enumerate}
\item Equation \eqref{HE-main-hh} has a unique solution  for every $T>0$;
 \item For every $t>0$, $u(t)\in L_2\big(\Omega; \tilde{H}^{1-\gamma}\big)$ and
$$
u(t)=\sum_{k=1}^{\infty} e^{-\nu\lambda_k^2t}\varphi_k\mfk{h}_k + \sum_{k=1}^{\infty} \tilde{u}_k(t) \mfk{h}_k,
$$
where and $\tilde{u}_k(t),\ k\geq 1,$ are independent Gaussian
random variables with mean zero and variance
$$
\bE \tilde{u}_k^2(t)=\frac{\sigma^2}{2\nu\lambda_k^2}\Big(1-e^{-2\nu\lambda_k^2 t}\Big);
$$
\item As $t\to +\infty$, the $\tilde{H}^{1-\gamma}$-valued random variables $u(t)$
converge weakly to $\sigma(2\nu)^{-1/2}\tilde{W}$;
\item Equation \eqref{HE-main-hh} is ergodic and the unique invariant measure is the distribution
of $\sigma(2\nu)^{-1/2}\tilde{W}$ on $\tilde{H}^{1-\gamma}$;
\item If $\varphi\wc\sigma(2\nu)^{-1/2}\tilde{W}$, then $u(t)\wc \sigma(2\nu)^{-1/2}\tilde{W}$ for all $t>0$;
\item If $\bE \varphi=0$, then, for each $t>0$, the measure generated by $u(t)$ on $\tilde{H}^{1-\gamma}$ is absolutely continuous
with respect to the measure generated by $\sigma(2\nu)^{-1/2}\tilde{W}$.
 \end{enumerate}
\end{theorem}

\section{Some Comments on the One-Dimensional Case}
\label{Sec:OneD}

In one space dimension, the Gaussian free field $\bar{W}=\bar{W}(x), \ x\in \mathcal{O}\subseteq \bR$, is a regular,
as opposed to  a generalized, process: $\bar{W}(x)$ is a zero-mean Gaussian random variable for each $x\in \cO$,

If $\mathcal{O}=(a,b)$ is a bounded interval, then
\begin{equation}
\label{GFF-int}
\bar{W}(x)=\sum_{k=1}^{\infty} \frac{\zeta_k}{ {\lambda_k}}\,\mfk{h}_k(x).
\end{equation}
 In \eqref{GFF-int},
 \begin{itemize}
 \item $\zeta_k,\ k\geq 1,$ are idd   standard Gaussian random variables;
 \item $\mfk{h}_k=\mfk{h}_k(x)$ and $-\lambda_k^2<0$,   $k\geq 1,$ are the normalized eigenfunctions and the  eigenvalues of the
 Laplacian on $(a,b)$ with suitable boundary conditions:
 $$
\mfk{h}_k''(x)=-\lambda_k^2\mfk{h}_k(x),\ \int_a^b\mfk{h}_k^2(x)\, dx=1,\ \int_a^b\mfk{h}_k(x)\mfk{h}_m(x)\, dx=0,\ k\not= m.
 $$
 \end{itemize}
 The series in \eqref{GFF-int} converges with probability one because, by \eqref{Weyl0},
$\lambda_k \sim c k$.
Moreover,
$$
 \mathbb{E}\bar{W}(x)=0,\ \mathbb{E}\Big(\bar{W}(x)\bar{W}(y)\Big)=\sum_{k=1}^{\infty} \frac{\mfk{h}_k(x)\mfk{h}_k(y)}{\lambda_k^2}
 =\Phi_{\mathcal{O}}(x,y),
$$
 where $\Phi_{\mathcal{O}}$ is Green's function of the Laplacian on $(a,b)$ with appropriate boundary conditions, which also shows that
 \eqref{GFF-int} is the {\em Karhunen-Lo\`{e}ve decomposition} of $\bar{W}$ \cite[Example 3.2.18]{LR-SPDE}.
 Zero boundary conditions $\mfk{h}_k(a)=\mfk{h}_k(b)=0$ imply $\bar{W}$ is a (multiple of a) Brownian bridge on $[a,b]$,
 whereas $\mfk{h}_k(a)=\mfk{h}'_k(b)=0$ imply $\bar{W}$ is a (multiple of a) standard Brownian motion. Of course,   \eqref{GFF-int} is
 a particular case of \eqref{GFF-bnd000} and is consistent with the
 general definition \eqref{GFF-Gen} of the Gaussian free field.

For unbounded intervals, the convention is somewhat different.

If  $\mathcal{O}=(0,+\infty)$, then $\bar{W}$ is { defined} as the  standard Brownian motion; this convention, in particular, means the
  boundary condition at $x=0$ is fixed and is equal to zero.

If $\mathcal{O}=\bR$, then the Gaussian free field on $\mathcal{O}$ is  defined by
\bel{GFF1-1}
\breve{W}(x)=
\begin{cases}
W(x),& x>0\\
V(-x),& x<0.
\end{cases}
\ee
 In \eqref{GFF1-1},
 $W$ and $V$ are independent standard Brownian motions. In particular, $\breve{W}$ is a zero-mean Gaussian process with
covariance $\rho(x,y)=\bE\big(\breve{W}(x)\breve{W}(y)\big)$ given by
\bel{GFF1-2}
\rho(x,y)=
\begin{cases}
\min(|x|,|y|), & xy>0,\\
0,& xy<0.
\end{cases}
\ee
Equality \eqref{GFF1-2} means that  $\breve{W}$  is not a Gaussian free field in the sense of the general
definition \eqref{GFF-Gen} but rather the two-sided standard Brownian motion, which also happens to be the
 {\em L\'{e}vy Brownian motion} on $\bR$; cf. \cite[Chapter VIII]{Levy}. Indeed,
\eqref{GFF1-2} implies
$$
\bE\big(\breve{W}(x)-\breve{W}(y)\big)^2=|x-y|,\ x,y\in \bR.
$$

An alternative description of $\breve{W}$ on $\bR$ is a random generalized function acting on $f\in \mathcal{S}(\bR)$ by
\bel{GFF1-3a}
\breve{W}[f]=\int_{-\infty}^{+\infty} f(x)\breve{W}(x)\, dx.
\ee
Then
\bel{GFF1-3b}
\bE\Big(\breve{W}[f]\breve{W}[g]\Big)=
\int_{-\infty}^{+\infty}\int_{-\infty}^{+\infty}f(x)g(y)\rho(x,y)\, dxdy.
\ee
Given a function $f\in \mathcal{S}(\bR)$,  define the function $F=F(x)$ by
\bel{GFF1-3d}
F(x)=
\begin{cases}
\ds \int_{-\infty}^x f(t)\, dt,&\ x<0;\\
 & \\
\ds -\int_x^{+\infty} f(t)\, dt,&\ x>0.
\end{cases}
\ee
By direct computation, the function $F$ is continuous except possibly at $x=0$, and so
$$
F'(x)=f(x)-\left(\int_{-\infty}^{+\infty} f(t)\,dt\right)\delta(x),
$$
where $\delta(x)$ is the point mass (Dirac delta function) at zero.
Moreover, if  $f\in \mathcal{S}(\bR)$, then, for all $p>0$,
\bel{GFF1-3d2}
\lim_{|x|\to +\infty} |x|^p \ |F(x)|=0.
\ee
We can integrate by parts in \eqref{GFF1-3a} using \eqref{GFF1-1}:
\bel{GFF1-3e}
\breve{W}[f]=\int_0^{+\infty} F(-x)\, dV(x)-\int_0^{+\infty}F(x)\, dW(x).
\ee
The transition from \eqref{GFF1-3a} to \eqref{GFF1-3e} essentially relies on \eqref{GFF1-3d2} and the equality $\breve{W}(0)=0$.

By \eqref{GFF1-3e}, we get an alternative form of \eqref{GFF1-3b}:
\bel{GFF1-4}
\bE\Big(\breve{W}[f]\breve{W}[g]\Big)=\int_{-\infty}^{+\infty} F(x)G(x)\, dx;
\ee
 the function $G$ is constructed from the function $g$ according to \eqref{GFF1-3d}.

 If
$$
\hat{f}(\xi)=\frac{1}{\sqrt{2\pi}}\int_{-\infty}^{+\infty} e^{-\mfi x\xi}f(x)\, dx
$$
is the Fourier transform of $f$, then, by direct computation,
\bel{GFF1-4a}
\hat{F}(\xi)=\frac{\hat{f}(\xi)-\hat{f}(0)}{\mfi \xi};
\ee
recall that
$$
\hat{f}(0)=\int_{-\infty}^{+\infty} f(x)\, dx.
$$
By \eqref{GFF1-4a} and the $L_2$ isometry of the Fourier transform, \eqref{GFF1-4} becomes
\bel{GFF1-5}
\bE\Big(\breve{W}[f]\breve{W}[g]\Big)=\int_{-\infty}^{+\infty} \frac{\big(\hat{f}(\xi)-\hat{f}(0)\big)\overline{\big(\hat{g}(\xi)-\hat{g}(0)\big)}}{|\xi|^2}\, d\xi;
\ee
as usual, $\overline{z}$ denotes the complex conjugate of $z$.
The integral on the right-hand side of \eqref{GFF1-5} converges as long as
the functions $\hat{f}$ and $\hat{g}$ are differentiable at zero, which is the case, for example, if
\bel{GFF1-55}
\int_{-\infty}^{+\infty} |xf(x)|\, dx< \infty,\ \ \int_{-\infty}^{+\infty} |xg(x)|\, dx< \infty,
\ee
and certainly  holds if $f,g\in \mathcal{S}(\bR)$.
With \eqref{GFF1-2} in mind,  condition \eqref{GFF1-55} is also sufficient for convergence of the integral on the right-hand side of
\eqref{GFF1-3b}.

Equality  \eqref{GFF1-5} confirms that  $\breve{W}$ from \eqref{GFF1-1} is not a Gaussian free field in the sense of
Definition \ref{def:GFF-12}.

\section{Summary and Further Directions}
\label{Sec:Sum}

Let $\mathcal{L}$ be a self-adjoint elliptic operator on a separable Hilbert space $H$.
 Under suitable conditions, we expect that, as $t\to +\infty$,  the
solution of the parabolic equation
$$
\frac{\partial u}{\partial t}=\mathcal{L} u +f
$$
to converge to the solution of the elliptic equation
$$
\mathcal{L} v = -f.
$$
The results of this paper show that, under some conditions, the
solution of the stochastic evolution equation
\bel{S1}
\frac{\partial u}{\partial t}=\mathcal{L}u + \dot{W},
\ee
driven by a cylindrical Brownian motion on $H$, converges in distribution  to the solution of
\bel{S2}
(-\mathcal{L})^{1/2} v = {V},
\ee
where ${V}$ is an isonormal Gaussian process on $H$.
In particular, we establish this convergence when $\mathcal{L}$ is the Laplace operator and
the solution of \eqref{S2} is  the Gaussian free field.
One could study equation \eqref{S1} with other operators $\mathcal{L}$ and driving processes $\dot{W}$, resulting in
 different limits coming out of equation \eqref{S2}. Beside purely mathematical interest, another motivation for this study is
scaling limits of  (mostly yet to be discovered) discrete models.


\def\cprime{$'$}
\providecommand{\bysame}{\leavevmode\hbox to3em{\hrulefill}\thinspace}
\providecommand{\MR}{\relax\ifhmode\unskip\space\fi MR }
\providecommand{\MRhref}[2]{%
  \href{http://www.ams.org/mathscinet-getitem?mr=#1}{#2}
}
\providecommand{\href}[2]{#2}

\end{document}